\documentclass[11pt]{article}
\usepackage{fullpage}
\usepackage{amsmath,amsthm,amsfonts,dsfont}
\usepackage{amssymb,latexsym,graphicx}
\usepackage{mathtools}
\usepackage{hyperref}
\usepackage{xcolor}
\hypersetup{
	colorlinks,
	linkcolor={red!75!black},
	citecolor={blue!75!black},
	urlcolor={blue!75!black}
}

\newtheorem{theorem}{Theorem}[section]

\newtheorem{claim}[theorem]{Claim}

\newtheorem{corollary}[theorem]{Corollary}

\newcommand{\R}{\ensuremath{\mathbb{R}}}
\newcommand{\Z}{\ensuremath{\mathbb{Z}}}

\newcommand{\lat}{\mathcal{L}}

\renewcommand{\epsilon}{\varepsilon}

\newcommand{\vol}{\mathrm{vol}}
\DeclareMathOperator{\spn}{span}

\renewcommand{\vec}[1]{\ensuremath{\boldsymbol{#1}}}
\newcommand{\basis}{\ensuremath{\mathbf{B}}}

\DeclarePairedDelimiter\inner{\langle}{\rangle}

\DeclarePairedDelimiter\floor{\lfloor}{\rfloor}

\mathtoolsset{centercolon}

\newcommand{\numberinball}[2]{N_{\leq #1}(#2)}
\newcommand{\numberonsphere}[2]{N_{= #1}(#2)}

\newcommand{\mystack}[2]{\substack{#1\\#2}}

\usepackage[hyperpageref]{backref}

\begin{document}
	
		\title{A simple proof of a reverse Minkowski theorem for integral lattices}
		\author{
			Oded Regev\thanks{Courant Institute of Mathematical Sciences, New York
				University. Supported by the Simons Collaboration on Algorithms and Geometry, by a Simons Investigator Award from the Simons Foundation, and by the National Science Foundation (NSF) under Grant No.~CCF-1320188. Any opinions, findings, and conclusions or recommendations expressed in this material are those of the authors and do not necessarily reflect the views of the NSF.}\\
			\and
			Noah Stephens-Davidowitz\thanks{Cornell University.  Supported in part by the NSF under Grant No.~CCF-2122230.}\\
			\texttt{noahsd@gmail.com}
		}
		\date{}
		\maketitle

    \begin{abstract}
        We prove that for any integral lattice $\lat \subset \R^n$ (that is, a lattice $\lat$ such that the inner product $\langle \vec{y}_1,\vec{y}_2 \rangle$ is an integer for all $\vec{y}_1, \vec{y}_2 \in \lat$) and any positive integer $k$,
        \[
            |\{ \vec{y} \in \lat \ : \ \|\vec{y}\|^2 = k\}| \leq 2 \binom{n+2k-2}{2k-1}
            \; ,
        \]
        giving a nearly tight reverse Minkowski theorem for integral lattices.
    \end{abstract}
  
		\section{Introduction}

		A lattice $\lat \subset \R^d$ is the set of integer linear combinations of linearly independent basis vectors $\basis = (\vec{b}_1,\ldots, \vec{b}_n) \in \R^{d \times n}$. 
		 We call $n$ the \emph{rank} of the lattice and $d$ the \emph{dimension}. We often implicitly assume that the lattice is \emph{full rank}, i.e., that $n = d$.
		
		A lattice $\lat$ is called an \emph{integral lattice} if $\inner{\vec{y}_1, \vec{y}_2} \in \Z$ for all $\vec{y}_1, \vec{y}_2 \in \lat$. Such lattices arise naturally in a number of contexts, from number theory to sphere packing to cryptography.

		We are interested in the largest number of short vectors in an integral lattice $\lat \subset \R^n$, i.e., we study how large \[\numberinball{k}{\lat} := |\{\vec{y} \in \lat : \|\vec{y}\|^2 \leq k\}|\] can be. For example, it is not hard to see that for any integral lattice,  $\numberinball{1}{\lat} \leq \numberinball{1}{\Z^n} = 2n+1$, i.e., $\Z^n$ has the maximal number of points with length at most $1$ among all integral lattices. The proof is easy but instructive. First, notice that integrality implies that $\lat$ cannot contain any non-zero points with length strictly less than $1$ since $\|\vec{y}\|^2 = \langle \vec{y}, \vec{y} \rangle \in \Z$ for $\vec{y} \in \lat$. Next, notice that any vectors $\vec{y}_1,\vec{y}_2 \in \lat$ of length $1$ must satisfy $\langle \vec{y}_1,\vec{y}_2\rangle \in \{-1,0,1\}$ by the Cauchy-Schwarz inequality. So, up to sign, all of the distinct vectors with length $1$ in $\lat$ must be pairwise orthogonal.  Of course, any collection of pairwise orthogonal non-zero vectors in $n$ dimensions has size at most $n$. Therefore, after accounting for $\vec{0} \in \lat$ and signs, we immediately see that $\numberinball{1}{\lat} \le 2n+1$.

While a tempting conjecture, it is \emph{not} true that $\numberinball{k}{\lat} \leq \numberinball{k}{\Z^n}$ for all $k \geq 1$. For example, the eight-dimensional lattice $E_8$ is integral with $240$ points of length $\sqrt{2}$ (and no points with length $1$), so that $\numberinball{2}{E_8} = 241$, whereas $\numberinball{2}{\Z^8} = 129$. (But see Eq.~\eqref{eq:tight_all_params} for a plausible unproven variant of this inequality that replaces the discrete point-counting function with a smooth approximation, namely, the Gaussian mass or lattice theta function.)
 	
  Our main result is an upper bound on $\numberinball{k}{\lat}$ that is not much larger than $\numberinball{k}{\Z^n}$.
		
		\begin{theorem}
			\label{thm:unimodular_RM}
			For any integral lattice $\lat \subset \R^n$ and positive integer $k$,
			\begin{align}
	   	\label{eq:bound_ball}
			\numberinball{k}{\lat} \leq 2 \binom{n + 2k-1}{2k-1}-1
				\; .
			\end{align}
			In fact,
			\begin{align}
	   	\label{eq:tighterbound_sphere}
			 \numberonsphere{k}{\lat} := |\{ \vec{y} \in \lat \ : \ \|\vec{y}\|^2 = k\}| \leq  2 \binom{n + 2k-2}{2k-1}
			 \; .
			\end{align}
    Moreover, for any constant positive integer $k$,
    \begin{align}\label{eq:constkbound}
        \numberinball{k}{\lat} \leq 2^{k} (k-1)! \cdot n^{k} + o(n^{k})
        \; ,
    \end{align}
    where $o(n^k)$ represents some function $f_k(n)$ such that $\lim_{n \to \infty} f_k(n)/n^k = 0$.
\end{theorem}
		
		Notice that for $k=1$, the right-hand side of Eq.~\eqref{eq:bound_ball} is $2n+1$, which is tight when $\lat = \Z^n$. 
  For $k=2$, we include a tight bound in Section~\ref{sec:small_radii}.
  For larger $k$, 
  we note (in Appendix~\ref{app:Zn_lower_bound}) that for $n \geq 6$, 
  \begin{equation}
      \label{eq:Zn_lower_bound}
    \numberonsphere{k}{\Z^n} \geq \binom{\floor{n/2}+k-1}{k}
    \; ,
  \end{equation}
  so that Eq.~\eqref{eq:bound_ball} is essentially tight up to a factor of two in both $n$ and $k$. Similarly, for constant $k$, Eq.~\eqref{eq:constkbound} is tight up to a multiplicative constant (depending on $k$).

		We also derive the following immediate corollary, which bounds the \emph{Gaussian mass} (or theta function) of an integral lattice $\lat \subset \R^n$, and in particular shows that most of the Gaussian mass is contributed by $\vec0$. Such bounds are useful, e.g., for using Fourier-analytic techniques to understand the geometry of the dual lattice (as in, e.g.,~\cite{banaszczykNewBoundsTransference1993}). We refer the reader to \cite{DRStrongReverse16,RSReverseMinkowski17} for more discussion of such bounds and their usefulness. 
		
		\begin{corollary}
		    \label{cor:gaussian}
		    For any integral lattice $\lat \subset \R^n$,
		    \[
		        \sum_{\vec{y} \in \lat} \exp(-\tau^* \|\vec{y}\|^2) \leq 1+ C/n
		        \; ,
		    \]
		    where $\tau^* := 2\log(2n)$ and $C>0$ is a universal constant.
		\end{corollary}

  This bound is quite tight since 
  \[
  \sum_{\vec{z} \in \Z^n} \exp(-\tau^* \|\vec{z}\|^2) 
  =
  \Big( \sum_{z \in \Z} \exp(-\tau^* z^2) \Big)^n 
  \geq 
  (1 + 2\exp(-\tau^*))^n
  \geq
  1 + 2n \exp(-\tau^*) 
  = 
  1+1/(2n) \; .
  \]

	    \subsection{Some context: reverse Minkowski theorems}
	     
	    One can view Theorem~\ref{thm:unimodular_RM} as a ``reverse Minkowski theorem,'' as we explain next. Minkowski's celebrated theorem~\cite{MinGeometrieZahlen10} says that a lattice with small determinant must have short non-zero vectors, where the lattice determinant is defined as $\det(\lat) := \det(\basis^T \basis)^{1/2}$. Most relevant for us is the following point-counting form of Minkowski's theorem due to Blichfeldt and van der Corput,\footnote{They actually showed the slightly stronger bound $\numberinball{k}{\lat} \geq 2\floor{2^{-n} \cdot \vol(\sqrt{k} B_2^n)}+1$ and considered arbitrary norms, not just $\ell_2$.  (See, e.g.,~\cite[Thm.~1 of Ch.~2, Sec.~7]{GLGeometryNumbers87}.)} 
		which says that a lattice with small determinant must have \emph{many} short points.
		
		\begin{theorem}[\cite{vanVerallgemeinerungMordellschen36}]
			\label{thm:minkowski}
			For any lattice $\lat \subset \R^n$ with $\det(\lat) \leq 1$ and $k > 0$,
			\begin{equation*}
			\numberinball{k}{\lat}
			\geq 2^{-n} \cdot \vol(\sqrt{k} B_2^n) 
			\approx \Big(\frac{\pi ek}{2n}\Big)^{n/2}
			\; . 
			\end{equation*}
		\end{theorem}
		
		One might hope to obtain a partial converse to Theorem~\ref{thm:minkowski}  by proving that lattices with determinant larger than one cannot have too many short points. However, this is easily seen to be false. For example, consider the lattice $\lat \subset \R^2$ spanned by the vectors $(1/T, 0)$ and $(0,T^2)$ for arbitrarily large $T \gg 1$, which has arbitrarily large determinant $\det(\lat) = T$ but arbitrarily many short points $\numberinball{k}{\lat} \geq 1 + \floor{2\sqrt{k} T}$. The issue here is that while the lattice itself has large determinant, it has a sublattice (spanned by $(1/T,0)$) that has small determinant.
		
 Dadush conjectured that this is in some sense the only obstruction~\cite{DadPersonalCommunication12}. Specifically, he conjectured a precise upper bound on $\numberinball{k}{\lat}$ for any lattice $\lat \subset \R^n$ such that $\det(\lat') \geq 1$ for all sublattices $\lat' \subseteq \lat$. Such a ``reverse Minkowski theorem'' was later shown to hold~\cite{RSReverseMinkowski17}, as in the below theorem. (See~\cite{RSReverseMinkowski17} for a slightly more refined statement.)
		
		\begin{theorem}[Reverse Minkowski theorem~\cite{RSReverseMinkowski17}]
			\label{thm:RM}
			For any lattice $\lat \subset \R^n$ such that $\det(\lat') \geq 1$ for all sublattices $\lat' \subseteq \lat$ and any $k > 0$,
			\begin{align}
				\numberinball{k}{\lat} \leq 2 \exp(\tau k)
				\; ,  \label{eq:revminkmain}
			\end{align}
			where $\tau := C\log^2(2n)$ for some explicit constant $C > 0$.
		\end{theorem}
		
		We refer the reader to~\cite{DRStrongReverse16,RSReverseMinkowski17} for more discussion of such inequalities and their applications.
     A major open question is whether one can take $\tau = C \log(2n)$ in Theorem~\ref{thm:RM}. (Using the example of $\Z^n$, it is easy to see that one cannot take $\tau$ to be any smaller.) 
     Notice that Theorem~\ref{thm:unimodular_RM} shows that this does in fact hold in the special case of integral lattices. 
     In particular, notice that integral lattices do in fact satisfy the condition that $\det(\lat') \geq 1$ for all sublattices $\lat' \subseteq \lat$, since (1) the determinant of an integral lattice must be the square root of a positive integer; and (2) a sublattice of an integral lattice is also integral. So, Theorem~\ref{thm:unimodular_RM} and Corollary~\ref{cor:gaussian} can be viewed as resolving this open question in the important special case of integral lattices. See 
     Section~\ref{sec:open} for more discussion of open questions.

\paragraph{Acknowledgments.} We thank Igor Balla, Surendra Ghentiyala, Stephen D. Miller, and Akshay Venkatesh for helpful discussions.

		\section{Proof of Theorem~\ref{thm:unimodular_RM}}
        \label{sec:proof_of_main_theorem}

		The proof of Theorem~\ref{thm:unimodular_RM} is very different from the (analytic and topological) proof of Theorem~\ref{thm:RM} from~\cite{RSReverseMinkowski17}, and surprisingly simple. The main (and essentially only) tool that we need is a theorem of Delsarte, Goethals, and Seidel~\cite{DGSBoundsSystems91} concerning the maximal number of lines whose pairwise angles are restricted to a few values. (See~\cite{KooNoteAbsolute76} for a short proof and~\cite{DGSSphericalCodes91} for more about such configurations of vectors.)
		
		\begin{theorem}[\cite{DGSBoundsSystems91}]
			\label{thm:few_angles}
			For any distinct unit vectors $\vec{v}_1,\ldots, \vec{v}_m \in \R^n$ satisfying
			\[
			\forall i,j,~
			|\inner{\vec{v}_i, \vec{v}_j}| \in A \cup \{0,1\}
			\; 
			\]
			for some finite set $A \subset (0,1)$,
			we must have
			\[
			m \leq 2  \binom{n+2|A|}{2|A| + 1}
			\; .
			\]			
		\end{theorem}

		To see why Theorem~\ref{thm:few_angles} is relevant, first notice that the case $A = \emptyset$ exactly captures our earlier observation that $\Z^n$ maximizes the number of vectors of length one.  More generally, consider the $m_k$ distinct vectors $\vec{y}_1,\ldots, \vec{y}_{m_k} \in \lat$ with squared length equal to some positive integer $k$, i.e., $\|\vec{y}_i\|^2 = k$. (All lattice vectors have integral squared length by assumption, since $\|\vec{y}\|^2 = \inner{\vec{y}, \vec{y}}$ and $\lat$ is integral.) 
		Clearly, $\inner{\vec{y}_i, \vec{y}_j}$ must be an integer, and by Cauchy-Schwarz, we see that the inner product must lie in a small set of integers.
		In particular,
		\[
		|\inner{\vec{y}_i, \vec{y}_j}| \in \{0,1,\ldots,k-1,k\}\;. 
		\] 
		Taking $\vec{v}_i := \vec{y}_i/\sqrt{k}$, we see that $\vec{v}_1,\ldots, \vec{v}_{m_k} \in \R^n$ are distinct unit vectors with
		\[
			|\inner{\vec{v}_i, \vec{v}_j}| \in A \cup \{0,1\}
			\; ,
		\]
		where $A := \{1/k, 2/k, \ldots, (k-1)/k\}$. Applying Theorem~\ref{thm:few_angles} and noticing that $|A| = k-1$ yields
		\begin{equation}
		\label{eq:bound_on_mk}
		m_k \leq 2 \cdot \binom{n+2|A|}{2|A| + 1} = 2\cdot  \binom{n + 2k-2}{2k-1}
		\; ,
		\end{equation}
		as claimed in Eq.~\eqref{eq:tighterbound_sphere} of Theorem~\ref{thm:unimodular_RM}.
		
		Next, notice that Eq.~\eqref{eq:tighterbound_sphere} implies Eq.~\eqref{eq:bound_ball}, because
		\[
		    1 + 2\sum_{j=1}^{k} \binom{n+2j-2}{2j-1} \leq 1 + 2\sum_{i=1}^{2k-1} \binom{n+i-1}{i} = 2 \binom{n+2k-1}{2k-1}-1 
		    \; .
		\]

For Eq.~\eqref{eq:constkbound}, we use a result from~\cite{keevashBoundsSphericalCodes2016,BDKSEquiangularLinesSpherical2018} improving upon Theorem~\ref{thm:few_angles} for sufficiently large $n$ in the case when $|A|$ is constant and $A \subset (\beta, 1)$ for some constant $\beta  > 0$. Here we present a special case of their main result that is sufficient for our purposes.

\begin{theorem}[{Special case of \cite[Theorem 1.4]{BDKSEquiangularLinesSpherical2018}}]
\label{thm:few_angles_2}
    For any constant positive integer $k$, if $\vec{v}_1,\ldots, \vec{v}_m \in \R^n$ are distinct unit vectors satisfying
    \[
       \forall i,j,~
       |\langle \vec{v}_i, \vec{v}_j \rangle| \in \{0,1/k,2/k,\ldots,(k-1)/k,1\}
        \; ,
    \]
    we must have
    \[
        m \leq 2^{k} (k-1)! \cdot n^k + o(n^k)
        \; .
    \]
\end{theorem}

This immediately implies that $\numberonsphere{k}{\lat}$ is bounded from above by the right-hand side of Eq.~\eqref{eq:constkbound}. 
The same bound also applies to $\numberinball{k}{\lat} = \sum_{i=0}^{k} \numberonsphere{i}{\lat}$ by noticing that $\numberonsphere{i}{\lat}=o(n^k)$ for $i < k$.

		\section{Proof of Corollary~\ref{cor:gaussian}}
		
		As in the previous section, we write $m_k$ for the number of lattice vectors of length $\sqrt{k}$. We have
		\[
		    \sum_{\vec{y} \in \lat} \exp(-\tau^* \|\vec{y}\|^2) =  \sum_{k=0}^\infty \exp(-\tau^* k) m_k = 1 + \sum_{k=1}^\infty (2n)^{-2k} m_k
		    \; .
		\]
		    Applying Eq.~\eqref{eq:bound_on_mk}, we see that the above is bounded by 
		    \[
		        1+2\sum_{k=1}^\infty (2n)^{-2k} \binom{n+2k-2}{2k-1} = 1+ \frac{1}{2n} \cdot \big( (1-1/(2n))^{-n} -(1+1/(2n))^{-n}\big) \leq 1+C/n
          \; ,
		    \]
		    where the equality follows from the negative binomial series identity
		    \[
		       \sum_{k=0}^\infty x^k \binom{n+k-1}{k} = (1-x)^{-n}
		        \; .
		    \]

\section{\texorpdfstring{Tight bounds for $k=2$}{Tight bounds for k=2}}
\label{sec:small_radii}

In this section, we
provide tight bounds for the case $k=2$.
To that end, we need the notion of a \emph{root system}.
 A root system is a finite set $\Phi \subset \R^n$ of non-zero vectors $\vec{y}_i \in \R^n$ such that (1) $\vec{y}_i- 2 \langle \vec{y}_i, \vec{y}_j \rangle \vec{y}_j/\|\vec{y}_j\|^2$ is in $\Phi$ for all $i,j$; (2) $2\langle \vec{y}_i, \vec{y}_j \rangle/ \|\vec{y}_j\|^2$ is an integer for all $i,j$; and (3) if $\vec{y}_i = \alpha\vec{y}_j$ for some scalar $\alpha$, then $\alpha = \pm 1$. It is easy to see that for any integral lattice, the set of vectors of length in $\{1,\sqrt{2}\}$ is a root system. There is a classification of all root systems, and from this classification we immediately derive the following theorem, which is tight for all $n$. See~\cite[Chapter 7]{ZhiIntroductionLatticeGeometry2016} for more discussion of root systems in the context of lattices. (See also~\cite{RootSystem2022}.) From this classification, we obtain the following.

\begin{theorem}
    \label{thm:root_system_thing}
    For any integral lattice $\lat \subset \R^n$, $\numberinball{2}{\lat} \le f(n)+1$ where 
    \begin{equation*}
        f(n) := \begin{cases}
            126 &n = 7\\
            240 + 2(n-8)^2 &8 \leq n \leq 11\\
            2n^2  &\text{otherwise}.
        \end{cases}
    \end{equation*}
Moreover, this bound is tight for all $n$. 
\end{theorem}
\begin{proof}
    A root system $\Phi$ is called \emph{reducible} if $\Phi$ can be written as a non-trivial disjoint union $\Phi = \Phi_1 \cup \cdots \cup \Phi_\ell$ where the sets $\Phi_i$ are pairwise orthogonal. Notice that in this case $\Phi_i$ is itself a root system. All root systems $\Phi \subset \R^n$ can be written as the disjoint union of \emph{irreducible} root systems $\Phi_1 \cup \cdots \cup \Phi_\ell$, where  $n \geq n_1 + \cdots + n_\ell$ for $n_i := \dim \spn \Phi_i$. The non-zero vectors with length at most $\sqrt{2}$ in an integral lattice form such a root system $\Phi$, and after accounting for zero, we have $\numberinball{2}{\lat} = 1+|\Phi| = 1 + \sum_i |\Phi_i|$.

    The irreducible root systems $\Phi_i$ are classified as follows. There are four infinite families of irreducible root systems, each of which satisfies $|\Phi_i| \leq 2n_i^2$. Then, there are five exceptional irreducible root systems: $G_2$, $F_4$, $E_6$, $E_7$, and $E_8$, where in each case the subscript represents the dimension of the (linear span of the) system. $G_2$ and $F_4$ do not arise as a subset of vectors of length $\{1,\sqrt{2}\}$ of an integral lattice, so they are not relevant to us. (In particular, when $G_2$ and $F_4$ are scaled so that the vectors have length at most $\sqrt{2}$, some of the resulting pairwise inner products are not integers.)
    $|E_6| = 72 = 2 \cdot 6^2$, so if $\Phi_i = E_6$, we still have the bound $|\Phi_i| \leq 2n_i^2$. The remaining two irreducible root systems have $|E_7| = 126 > 2 \cdot 7^2$ and $|E_8| = 240 > 2 \cdot 8^2$. Therefore, $|\Phi_i| \leq g(n_i)$, where $g(n_i) = 2n_i^2$ if $n_i \notin \{7,8\}$, $g(7) = 126$, and $g(8) = 240$.

    From the above, we have that
    \[\numberinball{2}{\lat}-1 \leq \max_{n_1+ \cdots + n_\ell = n} g(n_1) + \cdots + g(n_\ell) \; .
    \] 
    The result follows by noting that (1) $f(n) \geq g(n)$ for all $n$; and (2) $f(n)$ is superadditive, i.e., $f(n_1 + n_2) \geq f(n_1) + f(n_2)$ for all $n_1,n_2$ (which is trivial when $n_1,n_2 \notin \{7,\ldots, 11\}$ and can be verified by checking finitely many cases otherwise). Indeed, it follows that
    \[
        \max_{n_1 + \cdots + n_\ell=n} g(n_1) + \cdots + g(n_\ell) \leq \max_{n_1 + \cdots + n_\ell=n} f(n_1) + \cdots + f(n_\ell) \leq f(n)
        \; ,
    \]
    as needed.    

Tightness follows from the lattice $E_7$ ($n=7$), $E_8 \oplus \Z^{n-8}$ ($8 \le n \le 11$), or $\Z^n$ (otherwise). 
\end{proof}

\section{Open questions}
\label{sec:open}

We leave open the question of whether Theorem~\ref{thm:unimodular_RM} can be made tighter. In particular, Eq.~\eqref{eq:Zn_lower_bound} shows that Theorem~\ref{thm:unimodular_RM} is essentially tight up to a factor of two in both $k$ and $n$, and it is natural to ask whether this factor of two can be saved. Indeed, Eq.~\eqref{eq:constkbound} and Theorem~\ref{thm:root_system_thing} show how to save this factor of two for constant $k$, and it would be interesting to achieve something similar for all $k$. 
Such a strong bound would follow if one could show that $\Z^n$ maximizes the Gaussian mass, namely, that 
\begin{equation}
    \label{eq:tight_all_params}
    \sum_{\vec{y} \in \lat} \exp(-\tau \|\vec{y}\|^2) \leq \sum_{\vec{z} \in \Z^n} \exp(-\tau \|\vec{z}\|^2)
\end{equation}
for all $\tau > 0$ and any integral lattice $\lat$. If true, such an inequality would be very interesting on its own.  

We also ask whether it is possible to improve Theorem~\ref{thm:RM}, which achieves a weaker point-counting bound than our main theorem but applies to the more general class of \emph{substable} lattices, i.e., lattices that have no dense sublattices. One possible approach is to derive such a result from Theorem~\ref{thm:unimodular_RM} by a ``net'' argument.  
Specifically, show that the set of integral lattices is sufficiently dense in the set of substable lattices (under some metric) so that every substable lattice is not too far from an integral lattice, and argue that the strong bound of Theorem~\ref{thm:unimodular_RM} also applies to the nearby substable lattice.

  \appendix

  \section{Proof of \texorpdfstring{Eq.~\eqref{eq:Zn_lower_bound}}{a simple lower bound for integer points on a sphere}}
  \label{app:Zn_lower_bound}

  We will require the following consequence of Jacobi's four-square theorem, Jacobi's six-square theorem, and Jacobi's eight-square theorem. We note that the inequalities here are rather loose, as we have simply chosen bounds that lead to the simple expression in Eq.~\eqref{eq:Zn_lower_bound}.

  \begin{claim}
    \label{clm:Jacobi_proved_a_lot_of_stuff}
        For any positive integer $k$, $\numberonsphere{k}{\Z^6} \geq \binom{k+2}{2}$ and $\numberonsphere{k}{\Z^8} \geq \binom{k+3}{3}$. Furthermore, if $k$ is odd, then $\numberonsphere{k}{\Z^4} \geq 2k+3$.
  \end{claim}
  \begin{proof}
      Jacobi's four-square theorem says that for odd $k$, $\numberonsphere{k}{\Z^4} = 8\sum_{d | k} d$. In particular, $\numberonsphere{k}{\Z^4} \geq 8k \geq 2k + 3$, as needed.

      Jacobi's six-square theorem says that 
      \[
        \numberonsphere{k}{\Z^6} = 4\sum_{d|k} (k/d)^2(4 \chi(d) - \chi(k/d))
        \; ,
      \]
      where $\chi(m) = 0$ if $m$ is even, $\chi(m) = 1$ if $m = 1 \bmod 4$, and $\chi(m) = -1$ if $m = 3 \bmod 4$. 
      Therefore,
      \[
        \numberonsphere{k}{\Z^6} \geq 12 k^2 - 4\sum_{\mystack{d \geq 2}{d \text{ even}}} (k/d)^2 - 20 \sum_{\mystack{d \geq 3}{d \text{ odd}}} (k/d)^2 = 12 k^2 - (\pi^2/6) k^2 -20 (\pi^2/8-1) k^2 \geq \binom{k+2}{2}
      \]
      where the equality follows by recalling that $\sum_{d \geq 1} 1/d^2 = \pi^2/6$ and the last inequality follows by simply expanding out the binomial coefficient.

      Finally, Jacobi's eight-square theorem says that 
      \[
        \numberonsphere{k}{\Z^8} = 16 \sum_{d | k} (-1)^{d+k} d^3
\; .
      \]
      For odd $k$, all terms in the sum are positive, and we immediately see that $\numberonsphere{k}{\Z^8} \geq 16k^3$. For even $k$, we note that every odd divisor $d$ is half of an even divisor, so that
      \[
            \numberonsphere{k}{\Z^8} = 16 \sum_{\mystack{d | k}{d \text{ even}}} d^3 - 16\sum_{\mystack{d | k}{d \text{ odd}}} d^3 \geq 16 \sum_{\mystack{d | k}{d \text{ even}}} d^3 - 16\sum_{\mystack{d | k}{d \text{ even}}} (d/2)^3 \geq 14k^3
            \; .
        \]
        In either case, we have
        \[
            \numberonsphere{k}{\Z^8} \geq 14k^3 \geq \binom{k+3}{3}
            \; ,
        \]
        where again the last inequality follows by simply expanding out the binomial coefficient.
   \end{proof}
  
    \begin{claim}
        For any integer $n \geq 6$ and positive integer $k$,
    \[
    \numberonsphere{k}{\Z^n} \geq \binom{\floor{n/2}+k-1}{k}
    \; .
    \]
    \end{claim}
    \begin{proof}
        Notice that it suffices to prove the claim when $n$ is even, as the result when $n$ is odd then follows from the trivial inequality $\numberonsphere{k}{\Z^n} \geq \numberonsphere{k}{\Z^{n-1}}$.
        For even $n$, our proof is by induction on $n$, with base cases $n = 6$ and $n=8$ given by Claim~\ref{clm:Jacobi_proved_a_lot_of_stuff}. 

        For the induction step, we let $n > 8$ be even and we may assume that the claim holds for $n-4$. Then,
        \[
            \numberonsphere{k}{\Z^n} = %
            \sum_{i=0}^k \numberonsphere{k-i}{\Z^{n-4}}\numberonsphere{i}{\Z^{4}}\geq \numberonsphere{k}{\Z^{n-4}} + \sum_{\mystack{1 \leq i \leq k}{i \text{ odd}}} \numberonsphere{k-i}{\Z^{n-4}}\numberonsphere{i}{\Z^{4}} 
            \; .
        \]
        Applying the induction hypothesis and Claim~\ref{clm:Jacobi_proved_a_lot_of_stuff}, we see that
          \begin{align*}
    \numberonsphere{k}{\Z^n} 
    &\geq \binom{n/2+k-3}{k} + \sum_{\mystack{1 \leq i \leq k}{i \text{ odd}}} \binom{n/2+k-i-3}{k-i} (2i+3)\\
    &= \binom{n/2+k-3}{k} + \sum_{\mystack{1 \leq i \leq k}{i \text{ odd}}} \binom{n/2+k-i-3}{k-i} (i+1) + \sum_{\mystack{2 \leq j \leq k+1}{j \text{ even}}} \binom{n/2+k-j-2}{k-j+1} (j+1) \\
    &\geq \binom{n/2+k-3}{k} + \sum_{\mystack{1 \leq i \leq k}{i \text{ odd}}} \binom{n/2+k-i-3}{k-i} (i+1) + \sum_{\mystack{2 \leq j \leq k}{j \text{ even}}} \binom{n/2+k-j-3}{k-j} (j+1)\\
    &= \sum_{i=0}^k \binom{n/2 +k- i -3}{k-i}(i+1)\\
    &= \binom{n/2+k-1}{k}
    \; ,
  \end{align*}
  as claimed, where in the second line, we have applied the change of variables $j = i+1$, in the third line we have applied the inequality $\binom{m+1}{\ell+1} \geq \binom{m}{\ell}$, and the last equality can be proven combinatorially.\footnote{This last equality can be equivalently written as
  \[
    \binom{m+k}{m} = \sum_{i=0}^k \binom{m+k-i-2}{m-2}(i+1)
    \; ,
  \]
  which holds for all $m \geq 2$.
  To see this, consider choosing $m$ elements from the set $\{1,\ldots, m+k\}$ by first choosing the second-smallest element to be $i+2$ for some $0 \leq i \leq k$. There are then $i+1$ choices for the smallest element and $\binom{m+k-i-2}{m-2}$ choices for the remaining $m-2$ larger elements.}
    \end{proof}

\end{document}